\documentclass[reqno]{amsart}

\usepackage{cite}
\usepackage{amsmath}
\usepackage{amssymb}
\usepackage{amsthm}   
\usepackage{empheq}   
\usepackage{graphicx}
\usepackage{color}
\usepackage{appendix}
\usepackage{hyperref} 
\usepackage[centerlast]{caption}

\newtheorem{theorem}{Theorem}[section]
\newtheorem{lemma}[theorem]{Lemma}
\newtheorem{proposition}[theorem]{Proposition}

\theoremstyle{definition}
\newtheorem{definition}[theorem]{Definition}
\newtheorem{example}[theorem]{Example}

\theoremstyle{remark}
\newtheorem{remark}[theorem]{Remark}

\numberwithin{equation}{section}

\def\EE{{\mathcal{E}}}
\def\FF{{\mathcal{F}}}
\def\ss{\mathtt{s}}
\def\GGG{\mathbf{G}}
\def\SSS{\mathbf{S}}
\def\RRR{\mathbf{R}}

\def\tt{\mathtt{t}}
\def\sss{\overset{\circ}{\ss}}
\def\JJJ{\overset{\circ}{J}}

\begin{document}

\title[Weak convergence]{Weak convergence of regular Dirichlet subspaces}


\author{Liping Li*}
\address{Institute of Applied Mathematics, Academy of Mathematics and Systems Science, Chinese Academy of Sciences, Beijing 100190, China.}
\email{liping\_li@amss.ac.cn}
\thanks{*Corresponding author.}

\author{Toshihiro Uemura}
\address{Department of Mathematics, Kansai University, Yamate, Suita, Osaka, Japan.}
\email{t-uemura@kansai-u.ac.jp}

\author{Jiangang Ying**}
\address{School of Mathematical Sciences, Fudan University, Shanghai 200433, China.}
\email{jgying@fudan.edu.cn}
\thanks{**Research supported in part by NSFC grant 11271240.}






\begin{abstract}
In this paper we shall prove the weak convergence of the associated diffusion processes of regular subspaces with monotone characteristic sets for a fixed Dirichlet form. More precisely, given a fixed 1-dimensional diffusion process and a sequence of its regular subspaces, if the characteristic sets of regular subspaces are decreasing or increasing, then their associated diffusion processes are weakly convergent to another diffusion process. This is an extended result of \cite{SL15}.
\end{abstract}

\maketitle

\section{Introduction}\label{SEC1}

 Roughly speaking, for a fixed Dirichlet form, a regular Dirichlet subspace is its closed subspace with Dirichlet and regular properties. This terminology was first raised by M. Fukushima and J. Ying in \cite{FY03} \cite{FY04}, then they and their co-authors did a series of works on this topic, for example \cite{FFY05} \cite{FHY10} \cite{LY14} and \cite{LY15}. To introduce this conception, let $E$ be a locally compact separable metric space and $m$  a fully supported Radon measure on $E$. Then $L^2(E,m)$ is a Hilbert space, and its norm and inner product are denoted by $\|\cdot\|_m$ and $(\cdot, \cdot)_m$. The definitions of Dirichlet form and regularity are standard, and we refer them to \cite{CF12} and \cite{FOT11}. Further let $(\EE,\FF), (\EE',\FF')$ be two regular Dirichlet forms on $L^2(E,m)$. Then $(\EE',\FF')$ is called a \emph{regular Dirichlet subspace}, or a \emph{regular subspace} in abbreviation, if
 \[
 	\FF'\subset \FF,\quad \EE(u,v)=\EE'(u,v),\quad u,v\in \FF'. 
 \]
 We use $(\EE',\FF')\prec (\EE,\FF)$ to stand for that $(\EE',\FF')$ is a regular subspace of $(\EE, \FF)$. 

Recently, one of us with his co-author considered the Mosco convergence on regular subspaces of 1-dimensional diffusion process in \cite{SL15}. In those settings, the state space $E$ is $I=(a,b)$, an open interval, and $m$ is a fixed and fully supported Radon measure on $I$. Then $L^2(I,m)$ is a Hilbert space. The quadratic form 
\[
	(\mathcal{A},\mathcal{G}):=(\EE^{(\mathtt{s},m)}, \FF^{(\mathtt{s},m)}_0)
\]
 is a regular Dirichlet form on $L^2(I,m)$ associated with the scaling function $\mathtt{s}$, i.e. a strictly increasing and continuous function on $I$. More precisely, 
\begin{equation}
	\begin{aligned}
		&\FF^{(\mathtt{s},m)}_0:=\{u\in \FF^{(\mathtt{s},m)}:  u(a)~\text{or}~u(b)=0, \text{if}~a~\text{or}~b~\text{is}~\ss\text{-regular}\}, \\
		&\EE^{(\mathtt{s},m)}(u,v)=\frac{1}{2}\int_I \frac{du}{d\mathtt{s}}\frac{dv}{d\mathtt{s}}d\mathtt{s}, \quad u,v\in \FF^{(\mathtt{s},m)}_0,
	\end{aligned}
\end{equation}
where 
\begin{equation}
	\FF^{(\mathtt{s},m)}:=\bigg\{u\in L^2(I,m): u\ll \mathtt{s}, \frac{du}{d\mathtt{s}}\in L^2(I, d\mathtt{s})\bigg\}.
\end{equation}
We refer the above terminologies to \cite{CF12} \cite{FHY10} and \cite{SL15}. 
 Note that 
\[
	C_c^\infty\circ \ss:=\{u=\varphi\circ \ss: \varphi\in C_c^\infty(J)\}
\]
is a special standard core of $(\mathcal{A},\mathcal{G})$, where $J=\ss(I)$. Its associated Hunt process is denoted by $\mathbf{M}$. Hereafter, we always fix this regular Dirichlet form $(\mathcal{A},\mathcal{G})$ and its associated Hunt process $\mathbf{M}$. Let $\{(\EE^n,\FF^n):n\geq 1\}$ be a sequence of regular subspaces of $(\mathcal{A},\mathcal{G})$. In other words, for each $n$, 
\[
	\FF^n\subset \mathcal{G},\quad \EE^n(u,v)=\mathcal{A}(u,v),\quad u,v\in \FF^n. 
\]
Then for each $n$, there exists another scaling function $\ss_n$ such that (see Proposition~2.2 of \cite{SL15})
\[
	\ss_n\ll \ss, \quad \frac{d\ss_n}{d\ss}=0\text{ or }1,~d\ss\text{-a.e.}
\]
and 
\[
	(\EE^n,\FF^n)=(\EE^{(\ss_n,m)},\FF^{(\ss_n,m)}_0). 
\]
Set 
\begin{equation}
	G_n:=\bigg\{x\in I: \frac{d\ss_n}{d\ss}=1\bigg\}
\end{equation}
in the sense of $d\ss$-a.e., which is called the characteristic set of $(\EE^n,\FF^n)$. We already illustrated in \cite{LY15} that the characteristic set $G_n$ is an essential character of regular subspace, see also Lemma~2.3 of \cite{SL15}. The associated diffusion process of $(\EE^n,\FF^n)$ is denoted by $\mathbf{X}^n$. Let $(\EE,\FF)$ be another regular subspace of $(\mathcal{A},\mathcal{G})$, whose scaling function is $\ss_\infty$ and characteristic set is denoted by $G$. Its associated diffusion process is denoted by $\mathbf{X}$. We consider two situations: 
\begin{description}
\item[(D)] $G_n\downarrow G$ in the sense of $d\ss$-a.e., i.e. 
\[
	G_n\supset G_{n+1},\quad \bigcap_{n\geq 1}G_n=G,\quad d\ss\text{-a.e.}
	\]
	\item[(U)] $G_n\uparrow G$ in the sense of $d\ss$-a.e., i.e. 
\[
	G_n\subset G_{n+1},\quad \bigcup_{n\geq 1}G_n=G,\quad d\ss\text{-a.e.}
\]
\end{description}
In \cite{SL15}, the authors proved that in \textbf{(D)} and part of \textbf{(U)} (i.e. $G_n$ is open and $G=I$), $(\EE^n,\FF^n)$ converges  in the sense of Mosco as $n\rightarrow \infty$. 

In this paper, we shall extend the results of \cite{SL15}. 
Our main result stated in \S\ref{SECM} illustrates that in the cases of  \textbf{(D)} and \textbf{(U)}, under some mild conditions $(\EE^n,\FF^n)$ is not only  Mosco-convergent to $(\EE,\FF)$, but its associated diffusion process $\mathbf{X}^n$ is also convergent to $\mathbf{X}$ in the weak sense.

 Before proving the main result in \S\ref{SECP}, we shall deeply reconsider the Mosco convergence of \textbf{(D)} and \textbf{(U)} in \S\ref{SEC2}. We find if the characteristic sets are decreasing or increasing to another set, which is not necessarily a characteristic set, then their associated regular subspaces are Mosco-convergent to another Dirichlet form (may be in the wide sense), which is related to the limitation of characteristic sets, see Theorem~\ref{COR25} and Theorem~\ref{COR33}. We shall also give some interesting examples. For instance, a sequence of regular Dirichlet forms may converge to a ``zero'' Dirichlet form in the sense of Mosco, see Example~\ref{EXA26}. Particularly, for Case \textbf{(U)}, we shall prove the Mosco convergence in Theorem~\ref{COR33} without any other assumptions. In other word, we may delete all other conditions in Theorem~4.1 of \cite{SL15}.

\section{Main results}\label{SECM}

We first make the assumption:

\begin{description}
\item[(H1)] $\mathbf{X}^n$ and $\mathbf{X}$ are both conservative. 
\end{description}

\begin{remark}
For diffusion process $\mathbf{X}^n$, it is conservative if and only if neither $a$ nor $b$ is approachable in finite time, i.e. for $c\in I$,
\[
	\int_a^c m\left((x,c)\right)\ss_n(dx)=\infty,\quad \left(\text{resp. }\int_c^b m\left((c,x)\right)\ss_n(dx)=\infty \right).
\]
Thus in Case (\textbf{U}), (\textbf{H1}) is equivalent to that $\mathbf{X}^1$ is conservative, and in Case (\textbf{D}), it is equal to that $\mathbf{X}$ is conservative. Roughly speaking, we only need the conservativeness of the smallest Dirichlet space in the sequence. 

Note that this condition is mainly used to guarantee the Lyons-Zheng decomposition for all $T>0$ (without the restriction $T<\zeta$, where $\zeta$ is the life time of relevant diffusion process). 
\end{remark}

 Since $\mathbf{X}^n$ and $\mathbf{X}$ are both diffusion processes, we may assume that they share the same sample path space
\[
	\Omega=C\left([0,\infty), I\right)\subset C\left([0,\infty),\mathbf{R}\right). 
\]
In other words, $\Omega$ may be regarded as a subspace of $C\left([0,\infty),\mathbf{R}\right)$. Note that $C\left([0,\infty),\mathbf{R}\right)$ (and hence $\Omega$) is a separable metric space with its standard metric (see \cite{B99}). Define a class of trajectory functions on $C([0,\infty),\mathbf{R})$: 
\[
	Z_t(\omega)=\omega(t),\quad \omega\in C([0,\infty),\mathbf{R}),~ t\geq 0.
\]
Naturally, $\mathbf{X}^n$ (resp. $\mathbf{X}$) corresponds to a probability measure class $(\mathbf{P}^x_n)_{x\in I}$ (resp. $(\mathbf{P}^x)_{x\in I}$) on $\Omega$, and $Z=(Z_t)_{t\geq 0}$ may be regarded as their common trajectory function class. Let $\{\mu_n,\mu: n\geq 1\}$ be a class of probability measures on $I$, and define
\[
\begin{aligned}
	&\mathbf{P}^{\mu_n}_n(\cdot):=\int_{x\in I}\mu_n(dx)\mathbf{P}^x_n(\cdot), \\
	&\mathbf{P}^\mu(\cdot):=\int_{x\in I}\mu(dx)\mathbf{P}^x(\cdot).
\end{aligned}\]
They are still probability measures on $\Omega$. We make the following conditions on $\{\mu_n,\mu: n\geq 1\}$.

\begin{description}
\item[(H2)] \begin{itemize}
\item[(i)] $\{\mu_n:n\geq 1\}$ is tight. 
\item[(ii)] $\mu_n(dx)=g_n(x)m(dx), \mu(dx)=g(x)m(dx)$ and $g_n\rightarrow g$ in $L^1(I,m)$ and $L^2(I,m)$ as $n\rightarrow\infty$. 
\end{itemize}
\end{description}

\begin{remark}
	For example, if $\mu_n(dx)=\mu(dx)=g(x)m(dx)$ and $g\in L^2(I,m)$ (for instance, $g$ is bounded), then (\textbf{H2}) is satisfied. 
	
	The second term of \textbf{(H2)} is mainly used to prove the weak convergence of finite dimensional distributions in \S\ref{SEC32}. 
\end{remark}

We have the last assumption as follows:

\begin{description}
\item[(H3)] Let $\overset{\circ}{\ss}$ be the scaling function of the smallest Dirichlet space in the sequence (i.e. for Case (\textbf{U}), $\overset{\circ}{\ss}=\ss_1$; for Case (\textbf{D}), $\overset{\circ}{\ss}=\ss_\infty$). Assume
\[
	\overset{\circ}{\ss}\ll m
\]
and 
\[
	\varphi:=\frac{d\overset{\circ}{\ss}}{dm}
\]
is bounded. 
\end{description}

\begin{remark}
	For the classical case, i.e. $(\mathcal{A},\mathcal{G})=(\frac{1}{2}\mathbf{D},H^1(\mathbf{R}))$, in other words, $\mathbf{M}$ is 1-dim Brownian motion, we know that (\textbf{H3}) is always right. 
\end{remark}

\begin{theorem}\label{THM1}
For the cases \textbf{(D)} and \textbf{(U)},	under the conditions (\textbf{H1}), (\textbf{H2}) and (\textbf{H3}), $\mathbf{P}^{\mu_n}_n$ is weakly convergent to $\mathbf{P}^\mu$ as $n\rightarrow \infty$. 
\end{theorem}

We know from \cite{LY14} that if the characteristic set of regular subspace is open, we may give a beautiful and deep description about its structure. However, \cite{LY14} also pointed out not each regular subspace has an open characteristic set.
Fortunately, the above theorem tells us although we cannot directly describe the structure of regular subspace for some general case (the characteristic set is not open), we may find some other ``good'' regular subspaces, whose associated diffusion processes weakly converge to the bad one. 

For example, assume that $(\mathcal{A}, \mathcal{G})$ corresponds to the 1-dimensional Brownian motion, and $(\EE,\FF)$ is one of its regular subspaces with the characteristic set $G$ (may not be open). Because of the regularity of Lebesgue measure, we may find a sequence of open sets $\{G_n:n\geq 1\}$ such that $G_n\downarrow G$ a.e. It follows that (see \cite{LY14}) $G_n$ is still a characteristic set of some regular subspace, say $(\EE^n,\FF^n)$, of $(\mathcal{A},\mathcal{G})$, which can be described very well through the technique of \cite{LY14}. By Theorem~\ref{THM1}, the associated diffusion process of $(\EE^n,\FF^n)$ is weakly convergent to that of $(\EE,\FF)$. 

\section{Mosco convergence}\label{SEC2}

In this section, we shall first introduce some results on Mosco convergence of monotone Dirichlet spaces in \S\ref{SECMC}. Note that the discussions in \S\ref{SECMC} are valid for general settings, not only for the cases on the open interval $I$. Then in \S\ref{SECM2} we shall extend the main results of \cite{SL15} to more general situations.

\subsection{Mosco convergence of monotone Dirichlet spaces}\label{SECMC}

We need to point out that in this part $(\EE^n,\FF^n)$ and $(\mathcal{A},\mathcal{G})$ are general Dirichlet forms on $L^2(E,m)$ (not only the Dirichlet forms on $L^2(I,m)$ in \S\ref{SEC1}). 

In the context of U. Mosco \cite{M94}, the Mosco convergence may be defined for closed forms in the wide sense, i.e. the quadratic forms which satisfy all conditions of closed forms except for the denseness of  domains in $L^2(E,m)$. Next, we shall write down its specific definition for handy reference.  For any quadratic form $(\EE,\FF)$ on $L^2(E,m)$, we always extend the domain of $\EE$ to $L^2(E,m)$ by 
\[
	\EE(u,u)=\infty, \quad u\in L^2(E,m)\setminus \mathcal{F}. 
\]
Furthermore, we say $u_n$ converges to $u$ weakly in $L^2(E,m)$, if  for any $v\in L^2(E,m)$, $(u_n,v)_m\rightarrow (u,v)_m$ as $n\rightarrow \infty$, and strong convergence in $L^2(E,m)$ means $\|u_n-u\|_m\rightarrow 0$ as $n\rightarrow \infty$. 

\begin{definition}\label{DEF141}
Let $\{(\EE^n,\FF^n):n\geq 1\}$ be a sequence of closed forms in the wide sense and $(\EE, \FF)$ another closed form in the wide sense  on $L^2(E,m)$. Then $(\EE^n, \FF^n)$ is said to be convergent to $(\EE, \FF)$ in the sense of Mosco as $n\rightarrow \infty$, if
	\begin{itemize}
	 \item[(a)] for any sequence $\{u_n: n\geq 1\}$ of functions in $L^2(E,m)$, which is convergent to another function $u\in L^2(E,m)$ weakly, it holds that
	 \begin{equation}\label{141a}
	  	\liminf_{n\rightarrow \infty} \EE^n(u_n,u_n)\geq \EE(u,u);
	 \end{equation}
	 \item[(b)] for any function $u\in L^2(E,m)$, there always exists a sequence $\{u_n:n\geq 1\}$ of functions in $L^2(E,m)$, which is convergent to $u$ strongly as $n\rightarrow \infty$, such that
	 \begin{equation}\label{141b}
	  	\limsup_{n\rightarrow \infty} \EE^n(u_n,u_n)\leq \EE(u,u). 
	 \end{equation}
	\end{itemize}
\end{definition}

Note that every closed form in the wide sense possesses an associated semigoup on $L^2(E,m)$, which is not necessarily strongly continuous. Let $(T_t^n)_{t\geq 0}$ and $(T_t)_{t\geq 0}$ be the semigroups of $(\EE^n,\FF^n)$ and $(\EE,\FF)$. Then $(\EE^n,\FF^n)$ is convergent to $(\EE,\FF)$ in the sense of Mosco, if and only if for any $f\in L^2(E,m)$ and $t\geq 0$,  $T_t^n f$ is strongly convergent to $T_tf$ in $L^2(E,m)$. 

\subsubsection{Decreasing case}

We always fix a regular Dirichlet form $(\mathcal{A},\mathcal{G})$ on $L^2(E,m)$. A decreasing sequence of regular subspaces means a sequence of regular subspaces $\{(\EE^n,\FF^n):n\geq 1\}$ of $(\mathcal{A},\mathcal{G})$, which satisfies
\[
	\FF^1\supset \FF^2 \supset \cdots \supset \FF^n\supset \cdots. 
\]
Define 
\begin{equation}\label{EQ2FNF}
	\FF_\infty:=\bigcap_{n\geq 1}\FF^n,\quad \EE_\infty(u,v):=\mathcal{A}(u,v),\quad u,v\in \FF_\infty. 
\end{equation}
Clearly, $\FF_\infty$ may not be dense in $L^2(E,m)$, and $(\EE_\infty, \FF_\infty)$ may not be a real Dirichlet form. But it is actually a Dirichlet form in the wide sense, which means that it satisfies all conditions of Dirichlet form except for the denseness of $\FF_\infty$ in $L^2(E,m)$. We refer its specific definition to \S1.3 of \cite{FOT11}.

\begin{lemma}
	Let $(\EE_\infty,\FF_\infty)$ be the quadratic form defined by \eqref{EQ2FNF}. Then it is a Dirichlet form on $L^2(E,m)$ in the wide sense. 
\end{lemma}
\begin{proof}
Clearly, $(\EE_\infty,\FF_\infty)$ is a bilinear symmetric quadratic form on $L^2(E,m)$. Thus it suffices to prove the closeness and Dirichlet property of $(\EE_\infty,\FF_\infty)$. Set 
\[
	\EE_{\infty,1}(u,v):=\EE_\infty(u,v)+(u,v)_m,\quad u,v\in \FF_\infty. 
\] 
Assume that $\{u_k:k\geq 1\}$ is an $\EE_{\infty,1}$-Cauchy sequence in $\FF_\infty$. For any $n\geq 1$, since $\FF_\infty\subset \FF^n$ and $\EE^n|_{\FF_\infty \times\FF_\infty}=\EE_\infty$,  it follows that $\{u_k: k\geq 1\}$ is also an $\EE^n$-Cauchy sequence in $\FF^n$. Hence there is a function $v_n\in \FF^n$ such that  $\|u_k-v_n\|_{\EE^n_1}\rightarrow 0$ as $k\rightarrow \infty$. Particularly, $u_k$ is $L^2(E,m)$-convergent to $v_n$ as $k\rightarrow \infty$. But $n$ is arbitrary for the existence of $v_n$. That implies $v_1=v_2=\cdots =v_n=\cdots$. Denote this common function by $v$. Then $v\in \cap_{n\geq 1}\FF^n=\FF_\infty$ and 
\[
\EE_{\infty,1}(u_k-v,u_k-v)=\EE^1_1(u_k-v_1,u_k-v_1)\rightarrow 0
\]
as $k\rightarrow \infty$. Therefore, the closeness of $(\EE_\infty,\FF_\infty)$ is proved. Finally, we turn to prove the Dirichlet property of $(\EE_\infty,\FF_\infty)$. Let $u$ be arbitrary function in $\FF_\infty$ and $\psi$ an arbitrary normal contraction on $\mathbf{R}$ (i.e. for any $t,s \in \mathbf{R}$, $|\psi(t)|\leq |t|, |\psi(t)-\psi(s)|\leq |t-s|$). Since for any $n\geq 1$, $u\in \FF_\infty\subset \FF^n$, it follows that $\psi\circ u\in \FF^n$ and $\EE^n(\psi\circ u,\psi\circ u)\leq \EE^n(u, u)$. That implies 
\[
	\psi\circ u\in \bigcap_{n\geq 1}\FF^n=\FF_\infty, 
\]
and naturally, 
\[
	\EE_\infty(\psi\circ u,\psi\circ u)=\EE^n(\psi\circ u,\psi\circ u),\quad  \EE_\infty(u,u)=\EE^n(u,u).
\]
 Hence
\[
	\psi\circ u\in \FF_\infty, \quad \EE_\infty(\psi\circ u,\psi\circ u)\leq \EE_\infty(u,u). 
\]
That completes the proof. 
\end{proof}

\begin{remark}\label{RM22}
Set $\mathcal{C}_n:=\FF^n\cap C_c(E)$, which is a special standard core of $(\EE^n,\FF^n)$, and $\mathcal{C}:=\cap_{n\geq 1} \mathcal{C}_n$. One may easily check that $(\EE_\infty,\FF_\infty)$ is regular if and only if $\mathcal{C}$ is a special standard core of $(\EE_\infty,\FF_\infty)$. In fact, 
\[
	\FF_\infty\cap C_c(E)=\left(\cap_{n\geq 1} \FF^n\right)\bigcap C_c(E)=\bigcap_{n\geq 1}\left(\FF^n\cap C_c(E)\right)=\mathcal{C}.
\]

However, the regularity of $(\EE_\infty, \FF_\infty)$ is not always satisfied, and indeed an example that $(\EE_\infty, \FF_\infty)$ is a Dirichlet form in the wide sense will be shown in Example~\ref{EXA26}. 
Furthermore, Example~3.4 of \cite{SL15} provided an example, in which the Mosco limitation $(\EE_\infty, \FF_\infty)$ in Proposition~\ref{THM22} is a real Dirichlet form (not only in the wide sense) but not a regular one. On the other hand, the first part of Example~5.1 in  \cite{SL15} is an example to illustrate that $(\EE_\infty,\FF_\infty)$ may be a regular Dirichlet form. 
\end{remark}

The following lemma asserts that the sequence $\{(\EE^n,\FF^n):n\geq 1\}$ of regular subspaces is convergent to the quadratic form $(\EE_\infty, \FF_\infty)$ in the sense of Mosco. Indeed, the regularities of $\{(\EE^n,\FF^n):n\geq 1\}$ are not necessary for this fact. We refer a similar result to Theorem~3.1 of \cite{S78}. 

\begin{proposition}\label{THM22}
For a given regular Dirichlet form $(\mathcal{A},\mathcal{G})$ on $L^2(E,m)$, assume $\{(\EE^n,\FF^n):n\geq 1\}$ is a decreasing sequence of regular subspaces of $(\mathcal{A},\mathcal{G})$. The quadratic form $(\EE_\infty, \FF_\infty)$ is defined by \eqref{EQ2FNF}. Then $(\EE^n,\FF^n)$ is convergent to $(\EE_\infty, \FF_\infty)$ in the sense of Mosco as $n\rightarrow \infty$. 
\end{proposition}
\begin{proof}
We first prove the second term (b) of Definition~\ref{DEF141}. For any function $u\in L^2(E,m)$, let $u_n:=u$ for any $n\geq 1$. If $u\notin \FF_\infty$, then $\EE_\infty(u,u)=\infty$ and clearly
\[
	\limsup_{n\rightarrow \infty}\EE^n(u_n,u_n)\leq \EE_\infty(u,u). 
\]
Now assume that $u\in \FF_\infty$. Particularly, for any $n\geq 1$, $u_n=u\in \FF_\infty\subset \FF^n$. Thus $\EE^n(u_n,u_n)=\EE_\infty(u,u)$, which indicates
\[
	\limsup_{n\rightarrow \infty}\EE^n(u_n,u_n)=\limsup_{n\rightarrow \infty}\EE_\infty(u,u)=\EE_\infty(u,u). 
\]

Next, we turn to prove the first term (a) of Definition~\ref{DEF141}. Assume that $u_n$ is weakly convergent to $u$ in $L^2(E,m)$ as $n\rightarrow \infty$. Without loss of generality, we may assume that $u_n\in \FF^n$ for any $n\geq 1$. In fact, if $u_n\notin \FF^n$, then $\EE^n(u_n,u_n)=\infty$. When we drop $u_n$ from the sequence of functions, the limitation in the left side of \eqref{141a} will decrease. Moreover, if there are only finite $u_n$ such that $u_n\in \FF^n$, then this limitation must be $\infty$ and \eqref{141a} is clear. Fix an integer $N$, for any $n\geq N$, since $u_n\in \FF^n\subset \FF^N$, it follows that $\{u_n:n\geq N\}\subset \FF^N$. Note that a sequence of closed quadratic forms, which are all the same quadratic form, is actually convergent to itself in the sense of Mosco. That implies
\begin{equation}\label{EQ2NNN}
	\liminf_{n\geq N, n\rightarrow \infty}\EE^N(u_n,u_n)\geq \EE^N(u,u). 
\end{equation}
On the other hand, it follows from $u_n\in \FF^n$ that $\EE^N(u_n,u_n)=\EE^n(u_n,u_n)$. If for some integer $N$, $u\notin \FF^N$, then $\EE^N(u,u)=\infty$. From \eqref{EQ2NNN}, we obtain that
\[
	\liminf_{n\rightarrow \infty}\EE^n(u_n,u_n)=\liminf_{n\geq N, n\rightarrow \infty}\EE^N(u_n,u_n)\geq \infty \geq \EE_\infty(u,u). 
\]
If for any integer $N$, $u\in \FF^N$, then it follows that $u\in \cap_{N\geq 1}\FF^N=\FF_\infty$. Clearly, for some integer $N$, $\EE_\infty(u,u)=\EE^N(u,u)$. Therefore, we can deduce from \eqref{EQ2NNN} that
\[
	\liminf_{n\rightarrow \infty}\EE^n(u_n,u_n)=\liminf_{n\geq N, n\rightarrow \infty}\EE^N(u_n,u_n)\geq \EE_\infty(u,u). 
\]
That completes the proof. 
\end{proof}

\subsubsection{Increasing case}

On the contrary, we shall consider the increasing sequences of regular subspaces. We still fix a regular Dirichlet form $(\mathcal{A},\mathcal{G})$ on $L^2(E,m)$. An increasing sequence of regular subspaces means a sequence of regular subspaces $\{(\EE^n,\FF^n):n\geq 1\}$ of $(\mathcal{A},\mathcal{G})$, which satisfies
\[
	\FF^1\subset \FF^2 \subset \cdots \subset \FF^n\subset \cdots. 
\]
Note that $\cup_{n\geq 1}\FF^n\subset \mathcal{G}$, which is a linear space. It follows from the closeness of $(\mathcal{A},\mathcal{G})$ that the quadratic form $(\mathcal{A}, \cup_{n\geq 1}\FF^n)$ is closable on $L^2(E,m)$. Denote the closure of $\cup_{n\geq 1}\FF^n$ in $\mathcal{G}$ relative to the inner product $\mathcal{A}_1$ by $\FF^\infty$, and define
\begin{equation}\label{EQ3EUV}
	\EE^\infty(u,v):=\mathcal{A}(u,v),\quad u,v\in \FF^\infty. 
\end{equation}
We first assert that $\left(\EE^\infty, \FF^\infty\right)$ is a regular Dirichet form on $L^2(E,m)$. 

\begin{lemma}\label{LM31}
The quadratic form $(\EE^\infty, \FF^\infty)$ given above is a regular Dirichlet form on $L^2(E,m)$. Furthermore, if $\mathcal{C}_n$ is a core of $(\EE^n,\FF^n)$ for each $n\geq 1$, then $\mathcal{C}:=\cup_{n\geq 1}\mathcal{C}_n$ is a core of $\left(\EE^\infty,\FF^\infty\right)$.
\end{lemma}
\begin{proof}
	We first claim that $\left(\EE,\cup_{n\geq 1}\FF^n\right)$ possesses the Dirichlet property. In fact, let $\psi$ be a normal contraction. Take any function $u$ in $\cup_{n\geq 1}\FF^n$. Then there exists an integer $N$ such that $u\in \FF^N$. Since $(\EE^N,\FF^N)$ is a Dirichlet form, it follows that $\psi\circ u\in \FF^N\subset \cup_{n\geq 1}\FF^n$ and 
	\[
		\EE(\psi\circ u,\psi\circ u)\leq \EE(u,u). 
	\]
That implies the Dirichlet property of $(\EE,\cup_{n\geq 1}\FF^n)$. Thus it follows from Theorem~3.1.1 of \cite{FOT11} that $\left(\EE^\infty,\FF^\infty\right)$ is a Dirichlet form on $L^2(E,m)$. 

Next, we turn to prove that $\mathcal{C}$ is a core of $(\EE^\infty,\FF^\infty)$.  Clearly, we only need to prove that $\mathcal{C}$ is dense in $\cup_{n\geq 1}\FF^n$ with the norm $\|\cdot \|_{\EE^\infty_1}$. Indeed, for any function $u\in \cup_{n\geq 1}\FF^n$, there is an integer $N$ such that $u\in \FF^N$. Thus we can take a sequence of functions $\{u_n: n\geq 1\}\subset \mathcal{C}_N$ such that $\|u_n-u\|_{\EE^N_1}\rightarrow 0$ as $n\rightarrow \infty$. Since $\FF^N\subset \FF^\infty$, it follows that $u_n\in \mathcal{C}$ and 
\[
	\|u_n-u\|_{\EE^\infty_1}\rightarrow 0,
\]
as $n\rightarrow \infty$. That completes the proof. 
\end{proof}

 For the increasing sequence $\{(\EE^n,\FF^n):n\geq 1\}$ of regular subspaces, we have an analogical result of Proposition~\ref{THM22}.  We also refer the relevant discussion to Theorem~3.2 of \cite{S78}. 

\begin{proposition}\label{THM32}
	Let $\{(\EE^n,\FF^n):n\geq 1\}$ be an increasing sequence of regular subspaces of $(\mathcal{A},\mathcal{G})$ on $L^2(E,m)$. The quadratic form $(\EE^\infty, \FF^\infty)$ is defined by \eqref{EQ3EUV}. Then $(\EE^n,\FF^n)$ is convergent to $(\EE^\infty, \FF^\infty)$ in the sense of Mosco as $n\rightarrow \infty$. 
\end{proposition}
\begin{proof}
	We first prove (a) of Definition~\ref{DEF141}. Take a sequence of functions $\{u_n:n\geq 1\}$ in $L^2(E,m)$, which is weakly convergent to another function $u\in L^2(E,m)$. Without loss of generality, we may assume that $u_n\in \FF^n$ for any $n\geq 1$. It follows from $u_n\in \FF^n\subset \FF^\infty$ that 
	\[
		\EE^n(u_n,u_n)=\EE^\infty(u_n,u_n). 
	\]
Since $(\EE^\infty,\FF^\infty)$ is convergent to itself in the sense of Mosco, we have
\[
	\liminf_{n\rightarrow \infty}\EE^\infty(u_n,u_n)\geq \EE^\infty(u,u). 
\]
Thus we can deduce that 
\[
	\liminf_{n\rightarrow \infty}\EE^n(u_n,u_n)=\liminf_{n\rightarrow \infty}\EE^\infty(u_n,u_n)\geq \EE^\infty(u,u). 
\]

Finally, we prove (b) of Definition~\ref{DEF141}. For any function $u\in L^2(E,m)$, if $u\notin \FF^\infty$, then clearly $\EE^\infty(u,u)=\infty$, and \eqref{141b} is surely right. Otherwise, if $u\in \FF^\infty$, then it follows from Lemma~\ref{LM31} that we may take a sequence of functions $\{u_n:n\geq 1\}\subset \mathcal{C}$ such that $\|u_n-u\|_{\EE^\infty_1}\rightarrow 0$ as $n\rightarrow \infty$. Moreover, we may assume that $u_n\in \FF^n$ for any $n\geq 1$. In fact, there always exists an increasing sequence of integers $\{k_n:n\geq 1\}$ such that $u_n\in \FF^{k_n}$. For any $n\geq 1$ and $k_{n-1}\leq l<k_n$, set 
\[
	v_l:= u_{n-1}. 
\]
Then for any $l\geq 1$, $v_l\in \FF^{n-1}\subset \FF^l$, and clearly $\|v_l-u\|_{\EE^\infty_1}\rightarrow 0$ as $l\rightarrow \infty$. Hence we can replace $\{u_n:n\geq 1\}$ by $\{v_l:l\geq 1\}$ to realize our assumption.  Consequently, it follows from $u_n\in \FF^n$ that $\EE^n(u_n,u_n)=\EE^\infty(u_n,u_n)$. Furthermore, $\|u_n-u\|_{\EE^\infty_1}\rightarrow 0$ implies
\[
	\limsup_{n\geq 1}\EE^n(u_n,u_n)=\limsup_{n\geq 1}\EE^\infty(u_n,u_n)=\lim_{n\geq 1}\EE^\infty(u_n,u_n)=\EE^\infty(u,u). 
\]
That completes the proof. 
\end{proof}

\subsection{Mosco convergence of the cases \textbf{(D)} and \textbf{(U)}}\label{SECM2}

In this section, let us turn back to the special cases of 1-dimensional irreducible diffusion processes on $I$. More precisely, as in \S\ref{SEC1}, $(\mathcal{A},\mathcal{G})=(\EE^{(\mathtt{s},m)}, \FF^{(\mathtt{s},m)}_0)$ is the regular Dirichlet form on $L^2(I,m)$ associated with the scaling function $\ss$. Then all regular subspaces of $(\mathcal{A},\mathcal{G})$ may be characterized by the class of scaling functions
\begin{equation}\label{EQ32SSI}
\begin{aligned}
	\mathbf{S}_\mathtt{s}(I):=\bigg\{\tilde{\mathtt{s}}: \tilde{\mathtt{s}}~&\text{is a strictly increasing and continuous function on}~I,\\& \tilde{\mathtt{s}}(e)=0,~\tilde{\mathtt{s}}\ll \mathtt{s},~\frac{d\tilde{\mathtt{s}}}{d\mathtt{s}}=0~\text{or}~1,~d\mathtt{s}\text{-a.e.}\bigg\},
	\end{aligned}
\end{equation}
where $e$ is a fixed point of $I$. Furthermore, $\mathbf{S}_\ss(I)$ has the following equivalent expression:
\[
	\mathbf{G}_\mathtt{s}(I):=\bigg\{ G\subset I: \int_{G\cap (c,d)}d\mathtt{s}>0,~\forall c,d\in I, c<d\bigg\}.
\]
In other words, 
\[
	\SSS_\ss(I)\rightarrow \GGG_\ss(I),\quad \tilde{\ss}\mapsto G_{\tilde{\ss}}:=\bigg\{x\in I: \frac{d\tilde{\ss}}{d\ss}(x)=1\bigg\}
\]
is a bijective mapping, see Lemma~2.3 of \cite{SL15}. The set $G_{\tilde{\ss}} \in \GGG_\ss(I)$ of $\tilde{\ss}\in \SSS_\ss(I)$ is the characteristic set of associated regular subspace, which is an equivalence class in the sense of $d\ss$-a.e. Note that the regular subspace associated with scaling function $\tilde{\ss}$ or characteristic set $G_{\tilde{\ss}}$ may be written as $\big(\EE^{(\tilde{\ss},m)},\FF^{(\tilde{\ss}, m)}_0\big)$.

\subsubsection{Case \textbf{(D)}}

Now we turn to consider the extended cases of \textbf{(D)}, in which the Mosco limitation is not necessarily a real Dirichlet form. Assume that $G$ is a subset of $I$, which may not belong to $\GGG_\ss(I)$. Set $F:=G^c$ and define
\begin{equation}\label{EQ2GUF}
\begin{aligned}
	&\bar{\mathcal{F}}:=\bigg\{u\in \FF^{(\ss,m)}: \frac{du}{d\ss}=0,~d\ss\text{-a.e. on }F\bigg\}, \\
	&\bar{\mathcal{E}}(u,v):=\EE^{(\ss,m)}(u,v)=\frac{1}{2}\int_I \frac{du}{d\ss}\frac{dv}{d\ss}d\ss,\quad u,v\in \bar{\mathcal{F}}. 
\end{aligned}
\end{equation}
We first assert that $(\bar{\mathcal{E}},\bar{\mathcal{F}})$ is a Dirichlet form in the wide sense. 

\begin{lemma}\label{LM24}
	Let $G\subset I$ and the quadratic form $(\bar{\mathcal{E}},\bar{\mathcal{F}})$ be given above. Then $(\bar{\mathcal{E}},\bar{\mathcal{F}})$ is a Dirichlet form in the wide sense on $L^2(I,m)$. Furthermore, $(\bar{\mathcal{E}},\bar{\mathcal{F}})$ is a real Dirichlet form if and only if $G\in \GGG_\ss(I)$. 
\end{lemma}
\begin{proof}
	For the first assertion, we only need to prove the closeness and Dirichlet property of $(\bar{\mathcal{E}},\bar{\mathcal{F}})$. Let $\{u_n:n\geq 1\}\subset \bar{\mathcal{F}}$ be an $\bar{\mathcal{E}}_1$-Cauchy sequence in $\bar{\mathcal{F}}$. Since $\bar{\mathcal{F}}\subset \FF^{(\ss,m)}$ and $(\EE^{(\ss,m)},\FF^{(\ss,m)})$ is closed, it follows that there is a function $u\in \FF^{(\ss,m)}$ such that $u_n\rightarrow u$ with the norm $\|\cdot\|_{\EE_1^{(\ss,m)}}$ as $n\rightarrow \infty$. Particularly, a subsequence of $\{u_n:n\geq 1\}$, which is still denoted by $\{u_n:n\geq 1\}$, satisfies
	\[
		\frac{du_n}{d\ss}\rightarrow \frac{du}{d\ss},\quad d\ss\text{-a.e.}
	\]
as $n\rightarrow \infty$. From $du_n/d\ss=0$, $d\ss$-a.e. on $F$ for any $n$, we can deduce that $du/d\ss=0$, $d\ss$-a.e. on $F$. Hence $u\in \bar{\mathcal{F}}$ and $\|u_n-u\|_{\bar{\mathcal{E}}_1}\rightarrow 0$ as $n\rightarrow \infty$. That implies the closeness of $(\bar{\mathcal{E}},\bar{\mathcal{F}})$. For the Dirichlet property of $(\bar{\mathcal{E}},\bar{\mathcal{F}})$, take a function $u=\varphi\circ \ss\in \bar{\mathcal{F}}$ and assume that $\psi$ is a normal contraction. Note that $du/d\ss=\varphi'\circ \ss$ and $d(\psi\circ u)/d\ss=(\psi\circ \varphi)'\circ \ss=\psi'\circ \varphi\circ \ss\cdot \varphi'\circ \ss$. Clearly, $\psi\circ u\in \FF^{(\ss,m)}$ and 
\[
\EE^{(\ss,m)}(\psi\circ u,\psi\circ u)\leq \EE^{(\ss,m)}(u,u). 
\]
 It follows from $du/d\ss=0$, $d\ss$-a.e. on $F$ that 
\[
	\frac{d\psi\circ u}{d\ss}=0, \quad d\ss\text{-a.e. on }F. 
\]
Hence $\psi\circ u\in \bar{\mathcal{F}}$ and $\bar{\mathcal{E}}(\psi\circ u, \psi\circ u)\leq \bar{\mathcal{E}}(u,u)$. 

Finally, if $G\in \GGG_\ss(I)$, then clearly $(\bar{\mathcal{E}},\bar{\mathcal{F}})$ is a real Dirichlet form (see Lemma~3.1 of \cite{SL15}). Otherwise, if $G\notin \GGG_\ss(I)$, then there is an open interval $(c,d)\subset I$ such that $(c,d)\subset F$, $d\ss$-a.e. In particular, any function $u$ in $\bar{\mathcal{F}}$ satisfies $du/d\ss=0$, $d\ss$-a.e. on $(c,d)$. It follows that $u$ is a constant on $(c,d)$, which indicates that $\bar{\mathcal{F}}$ cannot be dense in $L^2(I,m)$. In other words, $(\bar{\mathcal{E}},\bar{\mathcal{F}})$ is only a Dirichlet form in the wide sense. That completes the proof. 
\end{proof}

Assume that $\{G_n: n\geq 1\}\subset \GGG_\ss(I)$ is a decreasing sequence of sets, and for each $n$, $\ss_n$ is the associated scaling function of $G_n$. Note that $G_n$ is defined in the sense of $d\ss$-a.e., hence the decreasing sequence is also in the sense of $d\ss$-a.e. Define $(\bar{\EE}^n,\bar{\FF}^n):= (\EE^{(\ss_n,m)}, \FF^{(\ss_n,m)})$. Let 
\begin{equation}\label{EQ2GNG}
	G:=\bigcap_{n\geq 1} G_n,
\end{equation}
which is a subset of $I$, and the quadratic form $(\bar{\mathcal{E}},\bar{\mathcal{F}})$ is defined by \eqref{EQ2GUF} with respect to this subset $G$.  The following corollary may be regarded as an extension of Theorem~3.1 of \cite{SL15}.

\begin{theorem}\label{COR25}
	Let $\{G_n:n\geq 1\}$, $G$ and the quadratic form $(\bar{\mathcal{E}},\bar{\mathcal{F}})$ be given above. Then $\left(\bar{\EE}^n,\bar{\FF}^n\right)$ is convergent to $(\bar{\mathcal{E}},\bar{\mathcal{F}})$ in the sense of Mosco as $n\rightarrow \infty$.
\end{theorem}
\begin{proof}
	Since $G_n$ is decreasing relative to $n$, it follows from Lemma~3.1 of \cite{SL15} that 
	\[
		\bar{\FF}^1\supset \bar{\FF}^2\supset \cdots \supset \bar{\FF}^n\supset \cdots. 
	\]
Thus from Proposition~\ref{THM22}, we know that it suffices to prove 
\[
	\bigcap_{n\geq 1}\bar{\FF}^n=\bar{\mathcal{F}}. 
\]
In fact, by Lemma~3.1 of \cite{SL15}, $u\in \cap_{n\geq 1}\bar{\FF}^n$, if and only if $u\in \mathcal{G}$ and $du/d\ss=0$, $d\ss$-a.e. on $G_n^c$ for any $n\geq 1$, which implies that
\[
	\frac{du}{d\ss}=0,\quad d\ss\text{-a.e. on }\bigcup_{n\geq 1} G_n^c= \left(\bigcap_{n\geq 1}G_n\right)^c=G^c. 
\]
It follows from \eqref{EQ2GUF} that $u\in \cap_{n\geq 1}\bar{\FF}^n$ is equivalent to that $u\in \bar{\mathcal{F}}$. That completes the proof. 
\end{proof}

\begin{remark}
	The Dirichlet form $(\bar{\EE}^n,\bar{\FF}^n)$ differs to $(\EE^n,\FF^n)$ in \S\ref{SEC1}. Indeed, $(\bar{\EE}^n,\bar{\FF}^n)$ is the active reflected Dirichlet space of $(\EE^n,\FF^n)$. The obstacle that we may meet when directly considering $(\EE^n,\FF^n)$ has been explained in \S3 of \cite{SL15}. Some similar results on $(\EE^n,\FF^n)$ were also given in Corollary~3.3 of \cite{SL15}. Particularly, we need to point out if \textbf{(H1)} is satisfied, then neither $a$ nor $b$ is $\ss_n$-regular ($\ss_\infty$-regular). Thus for Case \textbf{(D)}, $(\EE^n,\FF^n)$ converges to $(\EE,\FF)$ in the sense of Mosco. 
\end{remark}

Note that  \eqref{EQ2GUF} is a real Dirichlet form, if and only if $G\in \GGG_\ss(I)$. Thus the Mosco limitation $(\bar{\mathcal{E}}, \bar{\mathcal{F}})$ in Theorem~\ref{COR25} is not a real Dirichlet form, if and only if the set $G$ defined by \eqref{EQ2GNG} does not belong to $\GGG_\ss(I)$. An extreme and interesting example is as follows. 

\begin{example}\label{EXA26}
Assume that $I=\RRR$, $m$ is the Lebesgue measure on $\RRR$. Furthermore, $\ss$ is the natural scaling function on $\RRR$, i.e. $\ss(x)=x$, $x\in \RRR$. Then $\FF^{(\ss,m)}=\FF^{(\ss,m)}_0$, and $\left(\EE^{(\ss,m)},\FF^{(\ss,m)}_0\right)$ is exactly the associated Dirichlet form $\left(\frac{1}{2}\mathbf{D},H^1(\RRR)\right)$ of $1$-dimensional Brownian motion on $L^2(\RRR)$. 

Let $Q=\{r_k: k\geq 1\}$ be the set of all rational numbers. For any $n\geq 1$, define
\[
	G_n:=\bigcup_{k\geq 1}\left(r_k-\frac{1}{2^{k+1}\cdot n}, r_k+\frac{1}{2^{k+1}\cdot n}\right). 
\]
Since $Q$ is dense in $\RRR$, one may easily check that $G_n\in \GGG_\ss(\RRR)$. We denote the associated scaling function of $G_n$ by $\ss_n$. On the other hand, clearly $G_n$ is decreasing relative to $n$, and 
\[
	G:=\bigcap_{n\geq 1}G_n. 
\]
Since the Lebesgue measure of $G_n$ is not more than $1/n$, it follows that the Lebesgue measure of $G$ equals $0$.
Let $(\bar{\mathcal{E}},\bar{\mathcal{F}})$ be the quadratic form \eqref{EQ2GUF} relative to the above set $G$. Apparently, one may check that
\[
	\bar{\mathcal{F}}=\{\mathbf{0}\}. 
\]
Then the associated semigoup $(T_t)_{t\geq 0}$ of $(\bar{\mathcal{E}},\bar{\mathcal{F}})$ is exactly $T_tu=u$ for any $u\in L^2(\RRR)$ and $t\geq 0$. From Theorem~\ref{COR25}, we can obtain that $\left(\EE^{(\ss_n,m)},\FF^{(\ss_n,m)}\right)$ is convergent to $(\bar{\mathcal{E}},\bar{\mathcal{F}})$ in the sense of Mosco. Note that $\left(\EE^{(\ss_n,m)},\FF^{(\ss_n,m)}\right)$ is a regular subspace of $\left(\frac{1}{2}\mathbf{D},H^1(\RRR)\right)$, and we denote its semigroup by $(T_t^n)_{t\geq 0}$. Therefore,
\[
	\|T_t^n u-u\|_m\rightarrow 0,\quad n\rightarrow \infty,
\]
for any $u\in L^2(\RRR)$ and $t\geq 0$. 
\end{example}

\subsubsection{Case \textbf{(U)}}

In Theorem~4.2 of \cite{SL15}, we have already considered the Mosco convergence of Case \textbf{(U)}. However, the conditions on increasing characteristic sets are too strict, i.e. $\{G_n: n\geq 1\}$ is a sequence of increasing open characteristic sets such that 
\[
	\bigcup_{n\geq 1}G_n=I
\]
in the pointwise sense. We say a characteristic set is open, if one of its $d\ss$-versions is open and we let it be this open version. For more details, see Theorem~4.2 of \cite{SL15}. 

From now on, we shall prove this kind of Mosco convergence under general settings. 
Note that in Case \textbf{(U)}
\[
	\{G_n:n\geq 1\}\subset \GGG_\ss(I)
\] 
is a sequence of increasing characteristic sets (in the sense of $d\ss$-a.e.)
and
\[
	G=\bigcup_{n\geq 1}G_n. 
\]
Clearly, $G\in \GGG_\ss(I)$. The associated scaling functions of $G_n$ and $G$ are denoted by $\ss_n$ and $\ss_\infty$. Then $(\EE^n,\FF^n)=(\EE^{(\ss_n,m)},\FF^{(\ss_n,m)}_0)$ and $(\EE,\FF)=(\EE^{(\ss_\infty,m)},\FF^{(\ss_\infty,m)}_0)$ are both regular subspaces of $(\mathcal{A},\mathcal{G})$. Note that $G_n$ and $G$ are not necessarily open. 

\begin{theorem}\label{COR33}
	For Case \textbf{(U)}, the regular subspace $(\EE^n,\FF^n)$ is convergent to $(\EE,\FF)$ in the sense of Mosco as $n\rightarrow \infty$. 
\end{theorem}
\begin{proof}
Note that 
\[
	\ss_n\ll \ss_{n+1},\quad \frac{d\ss_n}{d\ss_{n+1}}=1_{G_n},\quad d\ss_{n+1}\text{-a.e.}
\]
It follows that $(\EE^n,\FF^n)$ is also a regular subspace of $(\EE^{n+1},\FF^{n+1})$ (Cf. Theorem~4.1 of \cite{FHY10}). In particular, 
\[
	\FF^1\subset \FF^2\subset \cdots \subset \FF^n\subset \cdots\subset \FF. 
\]
Note that $C_c^\infty\circ \ss_n$ is a special standard core of $(\EE^n,\FF^n)$ for each $n\geq 1$. Hence from Proposition~\ref{THM32}, we only need to prove that 
\[
	\bigcup_{n\geq 1} C_c^\infty\circ \ss_n
\]
is dense in $\FF$ with the norm $\|\cdot\|_{\EE_1}$. 

Set $J_n:=\ss_n(I)$ and $J_\infty :=\ss_\infty(I)$. Clearly, $J_n$ and $J_\infty$ are open intervals and 
\[
	J_1\subset J_2\subset \cdots\subset J_n\subset\cdots \subset J_\infty. 
	\]
We assert that $\cup_{n\geq 1}J_n=J_\infty$. Actually, $\cup_{n\geq 1}J_n\subset J_\infty$. On the contrary, take a point $x\in I$. Without loss of generality, we may assume that $x>e$, where $e$ is the fixed point in \eqref{EQ32SSI}. Since $0<\ss_\infty(x)<\infty$, it follows that
\[
	\ss_\infty(x)=\int_e^x 1_G(y)d\ss(y)=\int_e^x\lim_{n\rightarrow \infty} 1_{G_n}(y)d\ss(y)=\lim_{n\rightarrow \infty}\int_e^x 1_{G_n}(y)d\ss(y),
\]
which implies 
\begin{equation}\label{EQ3SXS}
	\ss_\infty(x)=\lim_{n\rightarrow \infty} \ss_n(x). 
\end{equation}
All above convergent sequences are increasing relative to $n$. Since $I$ is open, we can find a point $z\in I$ such that $x<z$. Similarly, we may deduce that
\[
	\ss_\infty(z)=\lim_{n\rightarrow \infty}\ss_n(z)>\ss_\infty(x). 
\]
Hence there is an integer $M$ such that $\ss_M(z)>\ss_\infty(x)$. In particular, $\ss_\infty(x)\in J_M$. Therefore, we can obtain that $J_\infty\subset \cup_{n\geq 1}J_n$. 

Note that $C_c^\infty\circ \ss_\infty$ is a special standard core of $(\EE,\FF)$. Take a function $u=\varphi\circ \ss_\infty\in C_c^\infty\circ \ss_\infty$, i.e. $\varphi\in C_c^\infty(J_\infty)$. Since the support $\text{supp}[\varphi]$ of $\varphi$ is compact, we may find a bounded closed interval $K\subset J_\infty$ such that \[
0\in K,\quad \text{supp}[\varphi]\subset K. 
\]
Clearly, $K$ is a compact subset of $J_\infty$. It follows from
\[
	K\subset J_\infty=\bigcup_{n\geq 1}J_n
\]
that there exists an integer $N$ such that $K\subset \cup_{1\leq n\leq N}J_n=J_N$. That means, for any $n\geq N$, we may regard $\varphi$ as a function in $C_c^\infty(J_n)$ by letting $\varphi=0$ on $J_n\setminus K$. Furthermore, define
\[
	u_n:=\varphi\circ \ss_n, \quad n\geq N,
\]
where $\varphi$ is the above function in $C_c^\infty(J_n)$.  Clearly, $u_n\in C_c^\infty\circ \ss_n$. 

Finally, we shall prove that 
\[
	\|u_n-u\|_{\mathcal{A}_1}\rightarrow 0,\quad n\rightarrow \infty. 
\]
For any $n\geq N$, since $\text{supp}[\varphi]\subset K, 0\in K$, one may easily check that the support of $\varphi \circ \ss_n$ and $\varphi\circ \ss_\infty$ are both subsets of $W:=\ss_N^{-1}(K)$, which is a compact subset of $I$. Then it follows from \eqref{EQ3SXS} and dominated convergence theorem that
\begin{equation}
\begin{aligned}\label{EQ3NIU}
	\lim_{n\rightarrow \infty} &\int_I \left(u_n(x)-u(x)\right)^2m(dx)\\&=\lim_{n\rightarrow\infty} \int_W \left(\varphi\left(\ss_n(x)\right)-\varphi\left(\ss_\infty(x)\right)\right)^2m(dx) \\
	&= \int_W\lim_{n\rightarrow\infty} \left(\varphi\left(\ss_n(x)\right)-\varphi\left(\ss_\infty(x)\right)\right)^2m(dx) \\
	&=0. 
\end{aligned}\end{equation}
On the other hand, 
\[
\begin{aligned}
	\mathcal{A}(u_n&-u,u_n-u) \\
	&=\frac{1}{2}\int_I \left(\frac{du_n}{d\ss}-\frac{du}{d\ss}\right)^2d\ss \\
		&=\frac{1}{2}\int_I \left(\varphi'\circ \ss_n \cdot \frac{d\ss_n}{d\ss}-\varphi'\circ \ss_\infty\cdot \frac{d\ss_\infty}{d\ss}\right)^2d\ss \\
		&\leq \int_I \left(\varphi'\circ \ss_n \cdot \frac{d\ss_n}{d\ss}-\varphi'\circ \ss_n \cdot \frac{d\ss_\infty}{d\ss}\right)^2d\ss  \\&\qquad \quad + \int_I \left(\varphi'\circ \ss_n \cdot \frac{d\ss_\infty}{d\ss}-\varphi'\circ \ss_\infty\cdot \frac{d\ss_\infty}{d\ss}\right)^2d\ss. 
\end{aligned}\]
Denote the two integrations in the last term of above inequality by $\Phi(n)$ and $\Psi(n)$.
For any $n\geq N$, since $\text{supp}[\varphi']\subset K$, we can deduce that the support of $\varphi'\circ \ss_n$ is also a subset of $W$. Moreover, $\varphi'\circ \ss_n$ is bounded by $\|\varphi'\|_\infty$ for any $n\geq N$.  Hence
\[
 	\Phi(n)\leq \|\varphi'\|_\infty^2 \int_W \left(1_{G_n}(y)-1_G(y)\right)^2 d\ss(y).
\]
Since $W$ is compact, it follows from the bounded convergence theorem that 
\[
\begin{aligned}
	\lim_{n\rightarrow \infty} \Psi(n)& \leq \|\varphi'\|_\infty^2\lim_{n\rightarrow \infty} \int_W \left(1_{G_n}(y)-1_G(y)\right)^2 d\ss(y) \\ &=\|\varphi'\|_\infty^2 \int_W\lim_{n\rightarrow \infty} \left(1_{G_n}(y)-1_G(y)\right)^2 d\ss(y) 
	\\&=0. 
\end{aligned}\]
For another integration $\Psi(n)$, we have
\[
\begin{aligned}
	\Psi(n)&=\int_I \left(\varphi'\circ \ss_n(y)-\varphi'\circ \ss_\infty(y)\right)^2\left(\frac{d\ss_\infty}{d\ss}\right)^2d\ss
		\\ &=\int_W \left(\varphi'\circ \ss_n(y)-\varphi'\circ \ss_\infty(y)\right)^2d\ss_\infty.
\end{aligned}\]
Similarly to \eqref{EQ3NIU}, we can obtain that 
\[
	\lim_{n\rightarrow \infty} \Psi(n)=0.
\]
That completes the proof. 
\end{proof}

Now, we can reconsider Lemma~\ref{LM31}, in which $\mathcal{C}$ is not asserted to be a special standard core of $(\EE^\infty, \FF^\infty)$.  In particular, if $\mathcal{C}_n=\FF^n\cap C_c(E)$, then one may easily check that 
\[
	\mathcal{C}=\bigcup_{n\geq 1}\mathcal{C}_n=\left(\cup_{n\geq 1}\FF^n\right) \bigcap C_c(E)
\]
is a special standard core of $(\EE^\infty, \FF^\infty)$. But what about the cases that $\mathcal{C}_n$ is another special standard core of $(\EE^n,\FF^n)$? More precisely, if $\mathcal{C}_n$ is a special standard core of $(\EE^n,\FF^n)$ for each $n\geq 1$, whether $\mathcal{C}:=\cup_{n\geq 1}\mathcal{C}_n$ is always a special standard core of $(\EE^\infty,\FF^\infty)$.  The following example based on the Dirichlet forms in Theorem~\ref{COR33} indicates that the answer is negative. 

\begin{example}
We use the same notations as those in Theorem~\ref{COR33}, i.e. 
\[
(\EE,\FF)=(\EE^{(\ss_\infty,m)},\FF^{(\ss_\infty,m)}_0), \quad (\EE^n,\FF^n)=(\EE^{(\ss_n,m)},\FF^{(\ss_n,m)}_0),\quad n\geq 1,
\]
where $\ss_\infty,\ss_n$ corresponds to the characteristic sets $G$ and $G_n$, $G_n$ is increasing to $G$ in the sense of $d\ss$-a.e. as $n\rightarrow \infty$. Without loss of generality,  further assume that for any $n\geq 1$, 
\begin{equation}\label{EQ3GNG}
	d\ss(G_{n+1}\setminus G_n)>0. 
\end{equation}
 Note that $\mathcal{C}_n:=C_c^\infty\circ \ss_n$ is a special standard core of $(\EE^n,\FF^n)$, and $(\EE^n,\FF^n)$ is a proper regular subspace of $(\EE^{n+1},\FF^{n+1})$. Furthermore, we have proved in Theorem~\ref{COR33} that 
\[
	\mathcal{C}=\bigcup_{n\geq 1}\mathcal{C}_n=\bigcup_{n\geq 1} C_c^\infty\circ \ss_n
\]
is a core of $(\EE,\FF)$. However, we shall prove that $\mathcal{C}$ is not a subspace of $C_c(I)$, and hence it is not a special standard core of $(\EE,\FF)$. 

Since the coordinate function $f(x)=x$ is locally in $C_c^\infty$, it follows that $\ss_n$ is locally in $\mathcal{C}_n$ for any $n\geq 1$. Denote all functions that locally belong to $\mathcal{C}$ by $\mathcal{C}_\mathrm{loc}$. We first prove
\begin{equation}\label{EQ3SSC}
	\ss_1+\ss_2\notin \mathcal{C}_\mathrm{loc}.
\end{equation}
In fact, because of \eqref{EQ3GNG}, we may take a relatively compact open interval $(c,d)$ such that 
\[
	e\in (c,d),\quad (c,d)\subset [c,d]\subset I
\]
and 
\begin{equation}\label{EQ3GCD}
d\ss(G_1\cap (c,d))<d\ss(G_2\cap (c,d))<d\ss(G_3\cap (c,d)),
\end{equation}
where $e$ is the fixed point in \eqref{EQ32SSI}. Assume that $\ss_1+\ss_2\in \mathcal{C}_\mathrm{loc}$. Then there is a function $u\in \mathcal{C}$ such that $\ss_1+\ss_2=u$ on $(c,d)$. Particularly, $u$ may be written as $u=\varphi\circ \ss_k$ for some integer $k$ and $\varphi\in C_c^\infty(J_k)$. Thus
\[
	\ss_1\circ \ss_k^{-1}(x)+\ss_2\circ \ss_k^{-1}(x)=\varphi(x), \quad x\in (\ss_k(c),\ss_k(d)). 
\]
If $k=1$, then we can deduce that $\ss_2\circ \ss_1^{-1}$ is smooth on $(\ss_1(c),\ss_1(d))$. Note that $\ss_2\circ \ss^{-1}_1$ is strictly increasing. That implies the Lebesgue-Stieljes measure $d\ss_2\circ \ss_1^{-1}$ is absolutely continuous with respect to the Lebesgue measure, which is denoted by $|\cdot|$, on $(\ss_1(c),\ss_1(d))$. Take a set $H=\ss_1(G^c_1\cap (c,d))$. Then it follows from \eqref{EQ3GCD} that
\[
	|H|=d\ss_1(G_1^c\cap(c,d))=\int_{G^c_1\cap (c,d)}1_{G_1}d\ss=0,
\]
whereas
\[
\begin{aligned}
	d\ss_2\circ \ss_1^{-1}(H)&=d\ss_2(G^c_1\cap (c,d))=d\ss(G_2\cap G_1^c\cap (c,d))\\&=d\ss(G_2\cap (c,d))-d\ss(G_1\cap (c,d))>0. 
\end{aligned}\]
That contradicts to $d\ss_2\circ \ss^{-1}_1\ll |\cdot|$. If $k=2$, similarly we can deduce that $\ss_1\circ \ss_2^{-1}$ is smooth on $(\ss_2(c),\ss_2(d))$. Note that $\ss_1\circ \ss_2^{-1}(0)=\ss_1(e)=0$. On the other hand, for any $x\in (\ss_2(c),\ss_2(d))$, we have
\begin{equation}\label{EQ3SSX}
\begin{aligned}
	\ss_1\circ \ss^{-1}_2(x)&=\int_{\ss_2^{-1}(0)}^{\ss_2^{-1}(x)}1_{G_1}(y)d\ss(y)=\int_{\ss_2^{-1}(0)}^{\ss_2^{-1}(x)}1_{G_1}(y)\cdot 1_{G_2}(y)d\ss(y) \\
	&= \int_{\ss_2^{-1}(0)}^{\ss_2^{-1}(x)}1_{G_1}(y)d\ss_2(y)=\int_0^x 1_{G_1}(\ss_2^{-1}(z))dz \\
	&=\int_0^x 1_{\ss_2(G_1)}(z)dz. 
\end{aligned}\end{equation}
Nevertheless, from \eqref{EQ3GCD}, we can obtain 
\[
\begin{aligned}
	|\ss_2(G_1)\cap \left(\ss_2(c),\ss_2(d)\right)|&=|\ss_2\left(G_1\cap (c,d)\right)|=d\ss\left(G_1\cap (c,d)\right) \\
	&<d\ss(G_2\cap (c,d))=d\ss_2((c,d)) \\
	&=|\left(\ss_2(c),\ss_2(d)\right)|.
\end{aligned}\]
That indicates that the derivative of $\ss_1\circ \ss_2^{-1}$ on $(\ss_2(c),\ss_2(d))$ is not continuous, which contradicts to the smoothness of $\ss_1\circ \ss_2^{-1}$ on $(\ss_2(c),\ss_2(d))$. If $k\geq 3$, without loss of generality, we may only consider the case $k=3$. Clearly, $\ss_1\circ \ss^{-1}_3+\ss_2\circ \ss^{-1}_3$ is smooth on $(\ss_3(c),\ss_3(d))$. Similarly to \eqref{EQ3SSX}, we can obtain that for any $x\in (\ss_3(c),\ss_3(d))$,
\[
	\ss_1\circ \ss^{-1}_3(x)+\ss_2\circ \ss^{-1}_3(x)=\int_0^x \left(1_{\ss_3(G_1)}(z)+1_{\ss_3(G_2)}(z)\right)dz. 
\] 
However, $\ss_3(G_1)\subset \ss_3(G_2)$ and it follows from \eqref{EQ3GCD} that 
\[
	|\ss_3(G_2)\cap \left(\ss_3(c),\ss_3(d)\right)|<|\left(\ss_3(c),\ss_3(d)\right)|,
\]
which also contradicts to the smoothness of $\ss_1\circ \ss^{-1}_3+\ss_2\circ \ss^{-1}_3$.  That completes the proof of \eqref{EQ3SSC}. In the mean time, since $\ss_n$ is locally in $\mathcal{C}_n\setminus \mathcal{C}_{n+1}$ and $\mathcal{C}_n\setminus \mathcal{C}_{n-1}$ for any $n\geq 1$ ($\mathcal{C}_0:=\emptyset$), we can also deduce that neither $\mathcal{C}_{n+1}\setminus \mathcal{C}_n$ nor $\mathcal{C}_n\setminus \mathcal{C}_{n+1}$ is empty. In other words, $\mathcal{C}_n$ is not increasing or decreasing relative to $n$. 

Now, we can prove that $\mathcal{C}$ is not a linear space. Indeed, since $\ss_1$ and $\ss_2$ are locally in $\mathcal{C}_1$ and $\mathcal{C}_2$ respectively, we may find two functions $u_1\in \mathcal{C}_1, u_2\in \mathcal{C}_2$ such that $\ss_1=u_1,\ss_2=u_2$ on $(c,d)$. If $\mathcal{C}$ is a linear space, then $u_1+u_2\in \mathcal{C}$, whereas
$\ss_1+\ss_2=u_1+u_2$ on $(c,d)$. That indicates $\ss_1+\ss_2\in \mathcal{C}_\mathrm{loc}$, which conduces to the contradiction. 

Furthermore, one may easily check that $\mathcal{C}$ satisfies all conditions of special standard core except for the linearity (surely, it is not an algebra either). 
\end{example}

\section{Proof of Theorem~\ref{THM1}}\label{SECP}

It is well known that we need to prove the weak convergence of finite dimensional distributions and the tightness of $\{\mathbf{P}^{\mu_n}_n:n\geq 1\}$. 

\subsection{Scale transform}
Let $\overset{\circ}{\ss}$ be the scaling function in \textbf{(H3)}. Clearly, $\overset{\circ}{\ss}$ is a strictly increasing and continuous function on $I$. Set 
\[
\overset{\circ}{J}:=\left\{	\overset{\circ}{\ss}(x):x\in I\right\}.
\]
Then $\overset{\circ}{J}$ is also an open interval of $\mathbf{R}$ and 
\[
	\overset{\circ}{\ss}: I\rightarrow \overset{\circ}{J}
\]
is a homeomorphism. Further set
\[
	\Omega_{\sss}:=\{\sss\circ \omega: \omega\in \Omega\}=C((0,\infty],\JJJ) \quad (\subset C([0,\infty), \mathbf{R}))
\]
and 
\[
	\eta_{\sss}: \Omega\rightarrow \Omega_{\sss},\quad \omega\mapsto \sss\circ \omega,
\]
where $\sss\circ \omega(t):=\sss(\omega(t)),~t\geq 0$. 
Let $\overset{\circ}{\tt}$ be the inverse of $\sss$, i.e. $\overset{\circ}{\tt}=\sss^{-1}$. Then $\overset{\circ}{\tt}$ is also a strictly increasing and continuous function, and the inverse of $\eta_{\sss}$ is actually 
\[
\eta_{\overset{\circ}{\tt}}: \Omega_{\sss}\rightarrow \Omega,\quad \omega\mapsto \overset{\circ}{\tt}\circ \omega.
\]
Since $\sss$ is bijective, it follows that $\eta_{\sss}$ is also bijective. Furthermore, we may also prove that $\eta_{\sss}$ is homeomorphic. 

\begin{lemma}\label{LM51}
	The mapping $\eta_{\sss}$ is homeomorphic.
\end{lemma}
\begin{proof}
We only need to prove $\eta_{\sss}$ is continuous. Let $\{\omega_n: n\geq 1\}$ be a sequence in $\Omega$ and $\omega$ another element in $\Omega$. Note that $\omega_n\rightarrow \omega$ in $\Omega$ if and only if for any fixed $T>0$, 
\begin{equation}\label{EQSTT}
\sup_{0\leq t\leq T}|\omega_n(t)-\omega(t)|\rightarrow 0,\quad n\rightarrow \infty. 
\end{equation}
In particular, there exists two constants $M_1, M_2$ such that $[M_1,M_2]\subset I$ and 
\[
	M_1\leq \omega_n(t)\leq M_2,\quad M_1\leq |\omega(t)|\leq M_2,\quad 0\leq t\leq T. 
	\]
Since $\sss$ is continuous on $I$, it follows that it is uniformly continuous on $[M_1,M_2]$. Thus for any $\epsilon>0$, there exists a constant $\delta>0$ such that for any $x,y\in [M_1,M_2]$ with $|x-y|<\delta$, 
\[
	|\sss(x)-\sss(y)|<\epsilon. 
\]
From \eqref{EQSTT}, we may deduce that there exists an integer $N$ such that for any $n>N$, 
\[
	\sup_{0\leq t\leq T}|\omega_n(t)-\omega(t)|<\delta. 
	\]
Hence 
\[
	\sup_{0\leq t\leq T}|\sss\circ \omega_n(t)-\sss\circ\omega(t)|<\epsilon,
\]
which implies 
\[
	\lim_{n\rightarrow \infty} \sup_{0\leq t\leq T}|\sss\circ \omega_n(t)-\sss\circ\omega(t)|=0.
\]
That completes the proof. 
\end{proof}

Define the following image measures on $\Omega_{\sss}$:
\[
	\mathbf{Q}^x_n:=\mathbf{P}^{\overset{\circ}{\tt}(x)}_n\circ \eta_{\sss}^{-1},\quad \mathbf{Q}^x:=\mathbf{P}^{\overset{\circ}{\tt}(x)}\circ \eta_{\sss}^{-1},\quad x\in \JJJ
\]
and 
\[
	\mathbf{Q}_n:= \mathbf{Q}^{\mu_n\circ \sss^{-1}}_n= \mathbf{P}^{\mu_n}_n\circ \eta_{\sss}^{-1},\quad \mathbf{Q}:= \mathbf{Q}^{\mu\circ \sss^{-1}}= \mathbf{P}^{\mu}\circ \eta_{\sss}^{-1}.
\]
Then 
\[
	(\Omega_{\sss}, Z=(Z_t)_{t\geq 0}, \mathbf{Q}^x_n)_{x\in \JJJ} \quad (\text{resp. }(\Omega_{\sss}, Z=(Z_t)_{t\geq 0}, \mathbf{Q}^x)_{x\in \JJJ})
\]
 is the associated coordinate-variable process of the spatial transformed process $\sss(\mathbf{X}^n)$ (resp. $\sss(\mathbf{X})$). 
 On the other hand, we also write $\mathbf{P}^{\mu_n}_n$ as $\mathbf{P}_n$ and $\mathbf{P}^\mu$ as $\mathbf{P}$ for short. Due to Lemma~\ref{LM51}, the following lemma is trivial and we omit its proof. 
 
 \begin{lemma}
 	The probability measure $\mathbf{P}_n$ is weakly convergent to $\mathbf{P}$ as $n\rightarrow \infty$ if and only if $\mathbf{Q}_n$ is weakly convergent to $\mathbf{Q}$ as $n\rightarrow \infty$. 
 \end{lemma}
 
 This lemma indicates that we may do the spatial transform induced by $\sss$ on $\mathbf{X}^n$ and $\mathbf{X}$ simultaneously. The results in \cite{LY15} showed that the relation of associated Dirichlet forms (say $(\EE^n,\FF^n)$ and $(\EE,\FF)$) is invariant. Furthermore, the basic assumptions \textbf{(H1), (H2)} and \textbf{(H3)} are still satisfied. In other words, without loss of generality, we may assume that $\sss$ is the natural scaling function, i.e. 
 \[
 	\sss(x)=x,~x\in I.
 \] 
 Let $(\overset{\circ}{\EE},\overset{\circ}{\FF})$ be the associated regular subspace of $\sss$. 
The direct corollary is $C_c^\infty(I)\subset \overset{\circ}{\FF}$, which is equivalent to that the coordinate function $f(x)=x$ is locally in $\overset{\circ}{\FF}$, i.e. $f\in \overset{\circ}{\FF}_{\mathrm{loc}}$. Since $\overset{\circ}{\FF}$ is the smallest Dirichlet space in the sequence, we may deduce that for any $n\geq 1$, 
\[
	f\in \FF^n_\mathrm{loc}, \quad f\in \FF_\mathrm{loc}. 
\]
Then we can write the Fukushima's decompositions (hence Lyons-Zheng decompositions) with respect to $f$ for $\mathbf{X}^n$ and $\mathbf{X}$, which is an essential technique to prove the tightness of $\{\mathbf{P}_n:n\geq 1 \}$. 

\subsection{Weak convergence of finite dimensional distributions}\label{SEC32}

From now on, we always assume that $\sss(x)=x,~x\in I$. 

We use $\mathbf{E}^x_n$ (resp. $\mathbf{E}^x$) to stand for the expectation with respect to $\mathbf{P}^x_n$ (resp. $\mathbf{P}^x$). For any $0\leq t_0<t_1<\cdots <t_k<\infty$ and $f_i\in b\mathcal{B}(I)\cap L^2(I,m)$ ($0\leq i\leq k$), define
\[
\begin{aligned}
	\mathbf{E}_n\left[ f_0(Z_{t_0})\cdots f_k(Z_{t_k})\right]&:=\int_{x\in I} \mathbf{E}_n^x\left[ f_0(Z_{t_0})\cdots f_k(Z_{t_k})\right]\mu_n(dx), \\
	\mathbf{E}\left[ f_0(Z_{t_0})\cdots f_k(Z_{t_k})\right]&:=\int_{x\in I} \mathbf{E}^x\left[ f_0(Z_{t_0})\cdots f_k(Z_{t_k})\right]\mu(dx). 
	\end{aligned}
\]
We need to prove the following proposition. 

\begin{proposition}
As $n\rightarrow \infty$,
\begin{equation}\label{EQENF}
	\mathbf{E}_n\left[ f_0(Z_{t_0})\cdots f_k(Z_{t_k})\right]\rightarrow \mathbf{E}\left[ f_0(Z_{t_0})\cdots f_k(Z_{t_k})\right]. 
\end{equation}
\end{proposition}
\begin{proof}
Let $(T^n_t)_{t\geq 0}$ and $(T_t)_{t\geq 0}$ be the semigroups of $\mathbf{X}^n$ and $\mathbf{X}$ respectively. Note that 
\[	
\mathbf{E}_n^x\left[ f_0(Z_{t_0})\cdots f_k(Z_{t_k})\right]=T_{t_0}^n\left(f_0\cdot T^n_{t_1-t_0}\left(\cdots \left(f_{k-1} \cdot T^n_{t_k-t_{k-1}}f_k \right) \cdots\right) \right)(x).
\]
On the other hand, since $(\EE^n,\FF^n)$ converges to $(\EE,\FF)$ in the Mosco sense (Cf. Theorem~\ref{COR25} and \ref{COR33}), we know that for any $t\geq 0$ and $f\in L^2(I,m)$, $\|T^n_tf-T_tf\|\rightarrow  0$ as $n\rightarrow \infty$. Thus we may deduce that
\[
	\mathbf{E}_n^\cdot\left[ f_0(Z_{t_0})\cdots f_k(Z_{t_k})\right]\rightarrow \mathbf{E}^\cdot \left[ f_0(Z_{t_0})\cdots f_k(Z_{t_k})\right]
\]
in $L^2(I,m)$. Then it follows from the second term of \textbf{(H2)} that \eqref{EQENF} holds. That completes the proof. 
\end{proof}

\subsection{Tightness}

Note that (see \cite{B99}) $\{\mathbf{P}_n:n\geq 1\}$ is tight if and only if 
\begin{itemize}
\item[(1)] $\liminf_{A\uparrow \infty} \mathbf{P}_n(|Z_0|\leq A)=1$;
\item[(2)] for any $\rho>0, T<\infty$,
\[
\lim_{\delta\downarrow 0}\limsup_{n\rightarrow \infty} \mathbf{P}_n\left(\sup_{0\leq s<t\leq T, t-s<\delta}|Z_t-Z_s|\geq \rho \right)=0. 
\]
\end{itemize}

\begin{proof}[Proof of ``tightness'']
For the first term (1), we have 
\[
	\mathbf{P}_n(|Z_0|\leq A)= \int_{x\in I}\mu_n(dx)\mathbf{P}_n^x(|Z_0|\leq A) =\mu_n( I\cap [-A,A]). 
\]
It follows from the first term of \textbf{(H2)} that (1) is right. Thus we only need to prove the second term in the above equivalent conditions. 

Fix $n$, $\rho>0$ and $T>0$. Since $f(x)=x$ is locally in $\FF^n$, we have the following Lyons-Zheng decomposition: for any $t,s>0$, 
\[
Z_t-Z_s=\frac{1}{2}(M^{[f]}_t-M^{[f]}_s)+\frac{1}{2}(M^{[f]}_{T-t}-M^{[f]}_{T-s})\circ r_T, \quad \mathbf{P}_n^m\text{-a.s.},
\]
where $r_T$ is the time reverse operator at $T$, i.e. $Z_t\circ r_T=Z_{T-t}$, and $M^{[f]}$ is the martingale part in Fukushima's decomposition with respect to $f$. In particular, for any $\mathcal{F}_T$-measurable function $F$, we have
\begin{equation}\label{EQEMN}
	\mathbf{E}^m_n\left(F\circ r_T\right)=\mathbf{E}^m_n(F). 
\end{equation}

We assert that the energy measure, denoted by $\mu_{\langle f \rangle}$, of $M^{[f]}$ equals $\sss(dx)$ (i.e. the Lebesgue measure on $I$). In fact, for any $u\in C_c^\infty(I)\subset b\FF^n$, 
\[
	\int_I u(x)\mu_{\langle f \rangle}(dx)=2\EE^n(f,f\cdot u)-\EE^n(f^2,u)=\int_I u(x) \left(\frac{df}{d\ss}\right)^2(x) \ss(dx).
\]
Note that $df/d\ss=d\sss/d\ss$ and $(d\sss/d\ss)^2=d\sss/d\ss$. 
Thus $\mu_{\langle f\rangle}(dx)=\sss(dx)$. Furthermore, it follows from \textbf{(H3)} that
\[
	\langle M^{[f]}\rangle_t=\int_0^t \frac{d\sss}{dm}(Z_s)ds,\quad \mathbf{P}^m_n\text{-a.s.}
\]
Therefore, for $m$-a.e. $x$, 
\[
\begin{aligned}
	\mathbf{P}^x_n&\left(\sup_{0\leq s<t\leq T, t-s<\delta}|Z_t-Z_s|\geq \rho \right)\\ &=\mathbf{P}^x_n\left(\sup_{0\leq s<t\leq T, t-s<\delta}\left|\frac{1}{2}(M^{[f]}_t-M^{[f]}_s)+\frac{1}{2}(M^{[f]}_{T-t}-M^{[f]}_{T-s})\circ r_T\right|\geq \rho \right)  \\
	&\leq \mathbf{P}^x_n\left(\sup_{0\leq s<t\leq T, t-s<\delta}\left|M^{[f]}_t-M^{[f]}_s\right|\geq \rho \right) \\ &\qquad \qquad + \mathbf{P}^x_n\left(\sup_{0\leq s<t\leq T, t-s<\delta}\left|\left(M^{[f]}_{T-t}-M^{[f]}_{T-s}\right)\circ r_T\right|\geq \rho \right) \\
	&=2\mathbf{P}^x_n\left(\sup_{0\leq s<t\leq T, t-s<\delta}\left|M^{[f]}_t-M^{[f]}_s\right|\geq \rho \right). 
\end{aligned}\]
The last equality is deduced from \eqref{EQEMN}. Clearly, 
\[
	\left\{ \sup_{0\leq s<t\leq T, t-s<\delta}\left|M^{[f]}_t-M^{[f]}_s\right|\geq \rho\right\}=\left\{ \sup_{0\leq s<t\leq T, t-s<\delta}\left| B_{\langle M^{[f]}\rangle_t}-B_{\langle M^{[f]}\rangle_s}\right|\geq \rho\right\},
\]
where $B=(B_t)_{t\geq 0}$ is a $\mathbf{P}^x_n$-Brownian motion. Note that $d\sss/dm$ is bounded and denote its bound by $C$. Particularly, $\langle M^{[f]}\rangle_t\leq C\cdot t$. Hence
\[
	\left\{ \sup_{0\leq s<t\leq T, t-s<\delta}\left|M^{[f]}_t-M^{[f]}_s\right|\geq \rho\right\}\subset \left\{ \sup_{0\leq s<t\leq C\cdot T, t-s<C\cdot\delta}\left| B_t-B_s\right|\geq \rho\right\}.
\]
Then we can deduce that
\[
	\mathbf{P}^x_n\left(\sup_{0\leq s<t\leq T, t-s<\delta}|Z_t-Z_s|\geq \rho \right)\leq 2 \mathbf{P}^x_n\left( \sup_{0\leq s<t\leq C\cdot T, t-s<C\cdot\delta}\left| B_t-B_s\right|\geq \rho\right),
\]
and the right side is actually independent of $n$. Finally we can easily check that 
\[
\begin{aligned}
	\lim_{\delta\downarrow 0}\limsup_{n\rightarrow \infty}& \mathbf{P}_n\left(\sup_{0\leq s<t\leq T, t-s<\delta}|Z_t-Z_s|\geq \rho \right) \\& \leq 2\lim_{\delta\downarrow 0}\limsup_{n\rightarrow\infty} \int_{x\in I} g_n(x)m(dx) \mathbf{P}^x_n\left( \sup_{0\leq s<t\leq C\cdot T, t-s<C\cdot\delta}\left| B_t-B_s\right|\geq \rho\right) \\
		&= 2\lim_{\delta\downarrow 0} \int_{x\in I} g(x)m(dx)\mathbf{P}^x_n\left( \sup_{0\leq s<t\leq C\cdot T, t-s<C\cdot\delta}\left| B_t-B_s\right|\geq \rho\right) \\
		&\rightarrow 0
	\end{aligned}
\]
as $n\rightarrow \infty$. That completes the proof. 
\end{proof}

\section*{Acknowledgement}
The authors would like to thank Professor P. J. Fitzsimmons from University of California at San Diego, for pointing out the relevant study \cite{S78} of two results in our paper (Proposition~\ref{THM22} and \ref{THM32}). The author at the first order also wants to thank Professor Ping He and Professor Minzhi Zhao for many helpful discussions.


\begin{thebibliography}{}

\bibitem{B99}
Billingsley, P.: \textsc{Convergence of probability measures.} John Wiley \& Sons, Inc., New York, Hoboken, NJ, USA (1999).

\bibitem{CF12}
Chen Z-Q, Fukushima M. \textsc{Symmetric Markov Processes, Time Change, and Boundary Theory.} Princeton NJ: Princeton University Press, 2012.

 \bibitem{FFY05}
Fang X, Fukushima M, Ying J. \textit{On regular Dirichlet subspaces of $H^1(I)$ and associated linear diffusions.} Osaka J Math., 2005, 42 (1): 27-41.

\bibitem{FHY10}
Fang X, He P, Ying J. \textit{Dirichlet forms associated with linear diffusions.} Chin Ann Math Ser B., 2010, 31 (4): 507-518.

\bibitem{FY03}
Fukushima M, Ying J. \textit{A note on regular Dirichlet subspaces.} Proc Amer Math Soc., 2003, 131 (5): 1607-1610.

\bibitem{FY04}
Fukushima M, Ying J. \textit{Erratum to: ``A note on regular Dirichlet subspaces" [Proc. Amer. Math. Soc. 2003, 131 (5): 1607--1610].} Proc Amer Math Soc., 2004, 132 (5): 1559-1560.

\bibitem{FOT11}
Fukushima M, Oshima Y, Takeda M. \textsc{Dirichlet Forms and Symmetric Markov Processes. extended.} Berlin: Walter de Gruyter \& Co., 2011.

\bibitem{KU97}
Kuwae, K., Uemura, T.: \textit{Weak convergence of symmetric diffusion processes.} Probab. Theory Related Fields. 109, 159-182 (1997).

\bibitem{LY14}
Li L, Ying J. \textit{On structure of regular subspaces of one-dimensional Brownian motion.} arXiv: 1412.1896, preprint, 2014.

\bibitem{LY15}
Li L, Ying J. \textit{Regular subspaces of Dirichlet forms.} In: {Festschrift Masatoshi Fukushima. In Honor of Fukushima's Sanju}: World Scientific, 2015, 397-420.

\bibitem{M94}
Mosco U. \textit{Composite media and asymptotic Dirichlet forms.} J Funct Anal.,  1994, 123 (2): 368-421.

\bibitem{S78}
Simon, B.:  \textit{A canonical decomposition for quadratic forms with applications to monotone convergence theorems.} J. Funct. Anal. 28, 377-385 (1978).

\bibitem{SL15}
Song X, Li L. \textit{Regular Dirichlet subspaces and Mosco convergence.} to appear in Chinese J. Contemp. Math. See also arXiv: 1505.00451, 2015. 

\bibitem{S98}
Sun, W.: \textit{Weak convergence of Dirichlet processes.} Sci. China Ser. A. 41, 8-21 (1998).

\bibitem{T89}
Takeda, M.: \textit{On a martingale method for symmetric diffusion processes and its applications.} Osaka J. Math. 26, 605-623 (1989).
\end{thebibliography}
\end{document}